\documentclass[10pt]{siamltex}
\usepackage{graphicx}
\usepackage{algorithm}
\usepackage[utf8]{inputenc}
\usepackage{mathrsfs}
\usepackage{amsmath}
\usepackage{amstext}
\usepackage{amsfonts}
\usepackage{marvosym}
\usepackage[psamsfonts]{amssymb}
\usepackage{latexsym}
\usepackage{bm}
\newcommand{\squareproof}{}
\newtheorem{remark}{Remark}

\begin{document}

\title{A primal dual mixed finite element method for inverse identification of the diffusion coefficient and its relation to the Kohn-Vogelius penalty method}

\author{Erik Burman\thanks{Department of Mathematics, University College London, London, UK--WC1E 6BT, UK, 
 {\tt e.burman@ucl.ac.uk}. }} 

	\maketitle
	%




\bigskip 
	
	\begin{abstract}
		We revisit the celebrated Kohn-Vogelius penalty method and discuss how to use it for the unique continuation problem where data is given in the bulk of the domain. We then show that the primal-dual mixed finite element methods for the elliptic Cauchy problem introduced in \cite{BLO18} (\emph{E. Burman, M. Larson, L. Oksanen, Primal-dual mixed finite element methods for the elliptic
              Cauchy problem, SIAM J. Num. Anal., 56 (6), 2018}) can be interpreted as a Kohn-Vogelius penalty method and modify it to allow for unique continuation using data in the bulk. We prove that the resulting linear system is invertible for all data. Then we show that by introducing a singularly perturbed Robin condition on the discrete level sufficient regularization is obtained so that error estimates can be shown using conditional stability.
              Finally we show how the method can be used for the identification of the diffusivity coefficient in a second order elliptic operator with partial data. Some numerical examples are presented showing the performance of the method for unique continuation and for impedance computed tomography with partial data.
	\end{abstract}
	

\section{Introduction}
The reconstruction of the coefficient in the second order operator in Poisson's problem is one of the most well-known inverse problems, known as the Calder\'on problem \cite{Cald80}. It appears in many applications, maybe the most well studied one is Electrical Impedance Tomography (EIT) which is important for instance in oil prospection or medical imaging. In the classical setting one assumes that the Dirichlet-Neumann (D-N) map is known and it has been shown that the diffusion coefficient is then uniquely determined. There are also stability estimates showing that the stability in the general case is no better than logarithmic. The literature on the Calder\'on problem is huge and we refer to the review paper \cite{Uhl09} for an overview of the mathematical theory. For work on computational methods for EIT we refer to \cite{KM90, Know98,BJL14, Rondi16, HKQ18} and for an introduction to finite element methods for the reconstruction of coefficients in elliptic problems we refer to \cite{Harr21}. In the above cited works data is available on the whole boundary. A case that is important in practice, but poorly understood, is that of partial data, or non-standard data. That is, the D-N map is known only on the subset of the boundary, only a small sample of the D-N map is known, or some data is known in the bulk instead of on the boundary. The stability theory for this case is less well developed and existing estimates very pessimistic \cite{Chou23}. In this case the forward problem of the inverse problem is ill-posed and any method based on solving forward problems using standard approaches will fail. Such inverse problems will be intrinsically linked to the unique continuation problem of the forward operator, since the partial data must allow for a continuation of the solution to the boundaries where data is not available. The unique continuation problem is an ill-posed linear problem. For example, assume that for given $\Omega \subset \mathbb{R}^d$,  $\gamma \in W^{1,\infty}(\Omega)$, with $\inf_{x \in \Omega} \gamma \ge \gamma_0 >0$, $u$ solves the equation
\begin{equation}\label{eq:toy}
\nabla \cdot (\gamma \nabla u) = 0 \mbox{ in } \Omega.
\end{equation}
It is then known that if $u$ and $\gamma \nabla u \cdot n$ are known on some subset of $\partial \Omega$ then $u$ is uniquely determined in all of $\Omega$ Similarly, if $u\vert_{\omega} = q$ is known for some subset $\omega \subset \Omega$ then $u$ is uniquely determined in $\Omega$. On the other hand, if the data $\gamma$ and the data in the bulk $q$ do not match a solution $u$ of \eqref{eq:toy} the solution may not exist. This is a complication for the determination of $\gamma$, since on the continuous level the unique continuation problem is known to admit a solution only for the exact $\gamma$, which is to be determined. The objective of the present study is to design and analyse a method that allows for non-standard data and can be used for the solution of the Calder\'on problem with partial data. Recently a new class of methods based of stabilized finite elements was introduced for such unique continuation problems \cite{Bur13,Bur14,Bur16,Bur17,BHL18, BNO19, BNO20}. The upshot of these methods is that error estimates can be proven that reflect the approximation order of the finite element method and the effect of stability and perturbations in a natural way, without resorting to a perturbed, regularized version of the continuous problem.
Here we handle the ill-posedness of the forward problem by using the finite element discretization introduced in \cite{BLO18} for the numerical approximation of the elliptic Cauchy problem. We will show that this method can be interpreted as a discrete realisation of the Kohn-Vogelius penalty method and discuss some variations. Firstly we will modify the method so that data can be introduced in the interior of the domain and show that the discrete system is invertible for all $\gamma$. Secondly we propose a  regularization on the boundary flux, through perturbation of the boundary condition, that alleviates the need for Tikhonov regularization in the interior of the domain in the method of \cite{BLO18}. We then extend the error analysis of \cite{BLO18} to allow for this new regularization, distributed data and non-constant and perturbed $\gamma$, setting the scene for its application to the approximation of the Calder\'on problem with partial data. These estimates reflect the approximation order of the finite element spaces, the impact of perturbations in data, as well as the stability of the inverse problem. We illustrate the theory in some numerical experiments on academic test cases, first for the unique continuation problem with distributed data and then on a Calder\'on problem with partial data.
\section{Notation and assumptions}
Below we let $\Omega \in \mathbb{R}^d$ denote some polygonal/polyhedral domain with boundary  $\partial \Omega$ and outward pointing normal $n$. 

We assume that $\gamma \in W^{1,\infty}(\Omega)$ has a lower bound larger than zero and that the solution of \eqref{eq:toy} has at least the additional regularity $u 
\in H^{2}(\Omega)$. The $L^2$-scalar product over some set $X$ will be defined by
\[
(v,w)_X := \int_X v w ~\mbox{d}X
\]
and the associated norm $\|w\|_X := (w,w)_X^{\frac12}$.
With some abuse of notation we will use the same notation for both scalar and vector quantities.
    \section{The Kohn-Vogelius penalty method}
   The Kohn-Vogelius method was introduced in the context of the problem of impedance computed tomography \cite{wexler1985impedance,KV87} where one wishes to recover the coefficient $\gamma(x) \in L^{\infty}(\Omega)$ of a second order elliptic problem given the Dirichlet to Neumann map
   \[
   \Lambda_\gamma:H^{\frac12}(\partial \Omega)\mapsto H^{-\frac12}(\partial \Omega)
   \]
   which takes $\varphi \in H^{\frac12}(\partial \Omega)$ to $\psi = (\gamma \nabla u_\varphi\cdot n)\vert_{\partial \Omega} \in H^{-\frac12}(\partial \Omega)$  where $u_\varphi$ solves $\nabla \cdot (\gamma \nabla u_\varphi) = 0$ with $u_\varphi\vert_{\partial \Omega} = \varphi$. In this situation, for a fixed $\gamma$ we can formulate two well posed problems, one with Dirichlet data and one with Neumann data, the solutions of which must coincide for the solution of the tomography problem. Given an approximation $\tilde \gamma$ of $\gamma$ We may write the two problems
   \begin{align}\label{eq:KV_D}
   \nabla \cdot (\tilde \gamma \nabla u_\varphi) & = 0 \mbox{ in }\Omega\\
   u_\varphi  & = \varphi \mbox{ on } \partial \Omega \nonumber
   \end{align}
   and
   \begin{align}\label{eq:KV_N}
   \nabla \cdot (\tilde \gamma \nabla u_\psi) & = 0 \mbox{ in }\Omega\\
   \tilde \gamma \nabla u_\psi \cdot n & = \psi \mbox{ on } \partial \Omega. \nonumber
   \end{align}
The Kohn-Vogelius method now consists of solving the optimization problem (see \cite{wexler1985impedance, KV87})
\[
\gamma = \mbox{argmin}_{\tilde \gamma} \|\tilde \gamma^{\frac12}\nabla (u_{\varphi} -  u_{\psi})\|^2_\Omega
\]
subject to \eqref{eq:KV_D} and \eqref{eq:KV_N}.
In this situation the Dirichlet and Neumann data are known everywhere on $\partial \Omega$ and we obtain a method where the optimization can be performed by solving a series of well-posed problems. In practice a common situation is that data is available only on some part of the boundary or in some subset of the domain. In both these cases the reconstruction relies on unique continuation for the differential operator and the Kohn-Vogelius method needs to be modified. For the sake of discussion let us consider the case where measurements of the solution are available in some subset $\omega \subset \Omega$. For a given $\gamma>0$ we know that $u\vert_\omega = q$ and
\[
\nabla \cdot (\gamma \nabla u) = 0 \mbox{ in }\Omega.
\]
Provided $q$ is indeed the data from a solution to the problem it is known that its continuation $u$ is unique. Nevertheless the problem is severely ill-posed in the sense of Hadamard. It seems that to handle the tomography problem with partial data some understanding of how the Kohn-Vogelius method can be applied to the unique continuation problem is necessary.
\subsection{The Kohn-Vogelius method for the unique contination problem}
Since no boundary data is available in this case it is natural to introduce a trace variable $\tilde g \in H^{\frac12}(\Omega)$ and consider the problems
   \begin{align}\label{eq:KV_UC1}
   -\nabla \cdot ( \gamma \nabla u) & = 0 \mbox{ in }\Omega\\
   u  & = \tilde g \mbox{ on } \partial \Omega \nonumber
   \end{align}
   and
   \begin{align}\label{eq:KV_UC2}
   -\nabla \cdot ( \gamma \nabla u_\omega) + \alpha \chi_{\omega} u_\omega & = \alpha \chi_{\omega} q\mbox{ in }\Omega\\
    u & = \tilde g \mbox{ on } \partial \Omega \nonumber
   \end{align}
   where $\chi_\omega$ denotes the characteristic function of the set $\omega$ and $\alpha \in \mathbb{R}^+$. We solve the optimization problem
   \begin{equation}\label{eq:gminim}
   g = \mbox{argmin}_{\tilde g \in H^{\frac12}(\partial \Omega)} \|\gamma^{\frac12}\nabla (u -  u_{\omega})\|^2_\Omega + \alpha \|u_\omega - q\|_\omega^2
   \end{equation}
   subject to \eqref{eq:KV_UC1}-\eqref{eq:KV_UC2}.
   One may then ask if this minimization problem admits a unique solution? Computing the Euler-Lagrange equations we see that the critical point must satisfy
\[
(\gamma \nabla (u_\omega - u),\nabla v)_\Omega  = 0 \mbox{ for all } v \in H^1_0(\Omega)
\]
   and
   \[
   (\gamma \nabla (u - u_\omega),\nabla v_\omega )_\Omega + \alpha ( u_\omega,v_\omega)_\omega = \alpha (q,v_\omega)_\omega \mbox{ for all } v_\omega \in H^1_0(\Omega)
   \]
     Using the constraint equations this reduces to
 \[
(\gamma \nabla  u_\omega,\nabla v)_\Omega  = 0 \mbox{ for all } v \in H^1_0(\Omega) \mbox{ and } ( u_\omega,v_\omega)_\omega = (q,v_\omega)_\omega \mbox{ for all } v_\omega \in H^1_0(\Omega)
\]    
and we see that in addition to the constraint equations \eqref{eq:KV_UC1} and \eqref{eq:KV_UC2},
$\nabla \cdot ( \gamma \nabla u_\omega) = 0$ must hold at the critical point. By using this additional condition we see that $u_\omega \vert_\omega = q$. It follows by unique continuation that the minimiser is unique. Nevertheless the minimisation problem is not stable under perturbations of the data $\gamma$ and $q$. Indeed the above argument fails if data are perturbed. It follows that the continuous model needs to be regularized before it can be useful. 
Finite element methods using Lavrientiev regularization based on this type of formulation was analysed in \cite{BGJ22} for the elliptic Cauchy problem. It was shown in \cite{belgacem2022uniqueness} that in the absence of regularization a naive choice of finite element spaces results in non-uniqueness of discrete solution.

Even after regularization a nested optimization seems inevitable since both $g$ and $\gamma$ have to be determined. Therefore we will follow a different approach. 
     \section{The mixed primal-dual finite element method}
     Instead of discretizing the continuous problem \eqref{eq:KV_UC1}-\eqref{eq:gminim} we will formulate the Kohn-Vogelius method diretcly in the discrete variables. To avoid the nested optimmization problem in $g$ and $\gamma$ implied by the discussion above we will drop the trace variable and instead focus on the Euler-Lagrange equations associated to the constrained system and show that, also in the absence of regularization and for all data, they admit a unique solution. To strengthen the norm of the discrete stability we then perturb the boundary condition in one of the sub-problems, leading to stability in the $H^1$-norm for the primal variable. Drawing on earlier results on the elliptic Cauchy problem and dual primal mixed methods we then show error bounds in $h$ and the size of data perturbation when perturbed data $\tilde q$, $\tilde \gamma$ and $\tilde f$ are used. To further distinguish this work from the result of \cite{BLO18} we here focus on the unique continuation problem with data in the bulk, as opposed to Dirichlet-Neumann data on the boundary. Nevertheless the discussion applies verbatim to the elliptic Cauchy problem as well.
     
We let $\{\mathcal{T}_h\}_h$ denote a family of matching quasi-uniform tesselations of $\Omega$ into shape regular simplices $T$, indexed by the mesh parameter $h = \max_{T \in \mathcal{T}} \mbox{diam}(T)$. We let $\mathbb{P}_p(T)$ denote the space of polynomials of degree less than or equal to $p$ on $T$ and define the space
\[
X_p := \{v \in L^2(\Omega): v\vert_T \in \mathbb{P}_p(T),\quad \forall T \in \mathcal{T}_h \}.
\]
We define $V_k := X_k \cap H^1(\Omega)$ and
 let $RT_p$ denote the Raviart-Thomas space of order $p$ on $\mathcal{T}_h$ \cite{RT77}, 
\[
RT_p:= \{q_h \in H_{div}(\Omega)\, : \, q_h\vert_T \in
\mathbb{P}_p(T)^d+ \vec x \, \mathbb{P}_p(T) \mbox{ for all } T \in \mathcal{T} \}.
\] 
 To write our method we will approximate the system \eqref{eq:KV_UC2} using a mixed formulation in $RT_p \times X_p$ and \eqref{eq:KV_UC1} using a primal formulation in $V_k$. To eliminate the variable $g$ the corresponding boundary condition will be dropped and instead the two systems are coupled on the boundary. Indeed \eqref{eq:KV_UC1} will take the trace of solution of \eqref{eq:KV_UC2} as Dirichlet data and \eqref{eq:KV_UC2} will take the normal flux of \eqref{eq:KV_UC1} as Neumann data. We consider a unique continuation problem with right hand side, $f \in L^2(\Omega)$.
For a given $\gamma>0$, $q\in H^1(\omega)$ we know that 
\begin{equation}\label{eq:UCf}
-\nabla \cdot (\gamma \nabla u) = f \mbox{ and } u\vert_\omega = q
\end{equation}
 
 We start by writing the Kohn-Vogelius functional (c.f. \cite[Equation (1.7b)]{KV87}) with constraint directly in the discrete setting. Consider for $\sigma_h \in RT_p$, $u_h \in V_k$ and $z_h \in X_p$,
 \begin{equation}
 \mathcal{L}_\gamma(\sigma_h,u_h,z_h) := \frac12 \|\gamma^{-\frac12} (\gamma \nabla u_h - \sigma_h)\|_\Omega^2 + \frac12 \alpha \|u_h - q\|_\omega^2 - (\nabla \cdot \sigma_h,z_h)_\Omega - (f,z_h)_\Omega.\label{eq:lagrange}
 \end{equation}

Our approximation will be given by the critical point of \eqref{eq:lagrange}. To see that this is indeed reasonable we compute the Euler-Lagrange equations. For all $v_h \in V_k$,
\begin{equation}\label{eq:EL1}
(\gamma \nabla u_h - \sigma_h,\nabla v_h)_\Omega + \alpha (u_h,v_h)_\omega = \alpha (q,v_h)_\omega.
\end{equation}
and for all $\tau_h \in RT_p$
\begin{equation}\label{eq:EL2}
( \nabla u_h - \gamma^{-1} \sigma_h,-\tau_h)_\Omega - (\nabla \cdot \tau_h,z_h)_\Omega = 0
\end{equation}
and finally for all $y_h \in X_p$,
\begin{equation}\label{eq:EL3}
-(\nabla \cdot \sigma_h, y_h)_\Omega = (f,y_h)_\Omega.
\end{equation}
So far this is valid for all choices of polynomial orders $p$ and $k$, to illustrate the relation between the present method and the classical Kohn-Vogelius method outlined in the previous section we will now consider the special case $k=p$. Then take the second term in the left hand side of equation \eqref{eq:EL1} and integrate by parts to see that, 
\[
-(\sigma_h,\nabla v_h)_\Omega = (\nabla \cdot \sigma_h,v_h)_\Omega - (\sigma_h \cdot n,v_h)_{\partial \Omega} = -(f,v_h)_\Omega - (\sigma_h \cdot n,v_h)_{\partial \Omega}.
\]
In the last equality we used that $V_p \subset X_p$ and the constraint equation \eqref{eq:EL3}.
It follows that we can write \eqref{eq:EL1},   find $u_h$ such that
\begin{equation}\label{eq:EL1b}
(\gamma \nabla u_h ,\nabla v_h)_\Omega + \alpha (u_h,v_h)_\omega = (\sigma_h \cdot n,v_h)_{\partial \Omega} + (f,v_h)_\Omega+\alpha (q,v_h)_\omega. \quad \forall v_h \in V_p
\end{equation}
Now observe that \eqref{eq:EL1b} coincides with the weak form of \eqref{eq:UCf} and the weakly imposed boundary condition $\gamma (\nabla u_h \cdot n)\vert_{\partial \Omega} = (\sigma_h \cdot n)\vert_{\partial \Omega}$. Here the data is integrated through the reaction term.

Similarly by integrating by parts in the first term of $\eqref{eq:EL2}$
\[
( \nabla u_h ,-\tau_h)_\Omega = (u_h,\nabla \cdot \tau_h)_\Omega - (u_h, \tau_h \cdot n)_{\partial \Omega}.
\]
Introducing the auxiliary variable $p_h = u_h - z_h$ we see that we may write \eqref{eq:EL2}-\eqref{eq:EL3} as find $\sigma_h,p_h \in RT_p \times X_p$ such that
\begin{align}\label{eq:EL2b}
(\gamma^{-\frac12} \sigma_h,\tau_h)_\Omega + (\nabla \cdot \tau_h,p_h)_\Omega & = (u_h, \tau_h \cdot n)_{\partial \Omega}\\
-(\nabla \cdot \sigma_h, y_h)_\Omega &= (f,y_h)_\Omega.\label{eq:EL3b}
\end{align}
This on the other hand is a mixed approximation of \eqref{eq:UCf} with the weakly imposed Dirichlet condition $p_h\vert_{\partial \Omega} = u_h\vert_{\partial \Omega}$ and without the data term. It follows that the Euler-Lagrange equations are consistent with the unique continuation problem \eqref{eq:UCf}. 

\subsection{Existence of unique discrete solution}
The equations \eqref{eq:EL1}-\eqref{eq:EL3} correspond to a square linear system and to prove that the linear system is invertible it is therefore enough to prove uniqueness of solutions. Assuming that two sets of solutions $U_1 = (u_1,\sigma_1,z_1)$ and $U_2 = (u_2,\sigma_2,z_2)$ solve \eqref{eq:EL1}-\eqref{eq:EL3} then if $U = U_1- U_2$, $U= (u_h,\sigma_h,z_h)$, is a solution to \eqref{eq:EL1}-\eqref{eq:EL3} with $q=f=0$.
We choose $v_h = u_h$ in \eqref{eq:EL1}, $\tau_h = \sigma_h$ in \eqref{eq:EL2} and $y_h = z_h$ in \eqref{eq:EL3} to obtain the bound
\[
\|\gamma^{-\frac12} (\gamma \nabla u_h - \sigma_h)\|_\Omega^2 +  \alpha \|u_h \|_\omega^2  = 0.
\]
It follows that $\gamma \nabla u_h = \sigma_h$ and $u_h\vert_{\omega} = 0$. This implies that equation \eqref{eq:EL2} reduces to
\[
(\nabla \cdot \tau_h,z_h)_\Omega = 0, \, \forall \tau \in RT_p. 
\]
By the inf-sup stability of the pair $RT_p\times X_p$  there exists $\tau_z \in RT_p$ such that $\nabla \cdot \tau_z  = z_h$ and hence $z_h = 0$.  As a consequence of $\gamma \nabla u_h = \sigma_h$ we have $u_h \in C^1(\Omega)$ and since $u_h$ is piecewise polynomial $u_h \in H^2(\Omega)$. It follows by \eqref{eq:EL3} that $\nabla \cdot (\gamma \nabla u_h) = 0$. However since $u_h\vert_\omega = 0$ we conclude by unique continuation that $u_h = 0$ and also $\sigma_h = 0$. We conclude that the discrete system \eqref{eq:EL1}-\eqref{eq:EL2} admits a unique solution for every $\gamma$ and for every right hand side $f$, $q$. This is somewhat surprising since the underlying continuous problem has a meaning only for compatible data. The explication is of course that we have no uniform a priori stability estimate on $u_h$ or $\sigma_h$ and due to the ill-posedness uniformity must be impossible to achieve for general data (otherwise one would  be able to pass to the limit in the equations to prove existence even in cases where the solution does not exist). Nevertheless, it is a useful result for the design of computational methods 
for the reconstruction of $\gamma$. Indeed in an optimization algorithm the gradient can always be computed, independent of any regularization.

In order to achieve error bounds for smooth solutions we need some tunable control of $u_h$. This was achieved in \cite{BLO18} by adding a Tikhonov regularization term on $u_h$ to the Lagrangian \eqref{eq:lagrange}, where the parameter was then chosen as a function of $h$ to achieve optimality in certain error estimates. Here we will suggest a different approach to stability using regularization of the boundary condition of the system \eqref{eq:EL1}-\eqref{eq:EL3}. 
\subsection{Regularization through perturbation of the boundary condition}
To give sufficient regularization on the discrete level to obtain an a priori bound on $u_h$ we introduce the relaxed boundary condition
\[
p_h = -\beta {\sigma_h} \cdot n + u_h
\]
in the equation \eqref{eq:EL2b}-\eqref{eq:EL3b}.
It is straightforward to show that this is achieved by adding the term 
\[
\frac12 \beta \|\sigma_h \cdot n\|^2_{\partial \Omega }
\]
to \eqref{eq:lagrange}.
This addition allows us to prove a discrete inf-sup condition in the following norm
\begin{align*}
|||w_h,\varsigma_h,x_h|||^2_\beta & := \beta \|w_h\|_{H^1(\Omega)}^2 + \alpha \|w_h\|_\omega^2 + \|\gamma^{-\frac12} (\gamma \nabla w_h - \varsigma_h)\|_\Omega^2 \\ & + \|x_h\|_{1,h}^2 + \beta \|\varsigma_h \cdot n\|_{\partial \Omega}^2+ \|h \nabla \cdot \varsigma_h\|_\Omega^2
\end{align*}
with 
\[
\|x_h\|_{1,h}^2 := \sum_{T \in \mathcal{T}_h} \left(\|\nabla x_h\|_T^2 +h_T^{-1} \|[x_h]\|_{\partial T}^2 \right)
\]
where $[\cdot]$ denotes the jump of a quantity over an interior face (the orientation is irrelevant) and for boundary faces we define $[y] = y$. 
It is well know that by the Poincaré inequality $|||\cdot|||_\beta$ is a norm also on $z_h$. Hence it is also a norm on $\sigma_h$.
To prove the inf-sup condition it is convenient to introduce the compact form, for $w_h,v_h \in V_k$, $\varsigma_h,\tau_h \in RT_p$ and $x_h,y_h \in X_p$ define,
\begin{align*}
A_\gamma[(w_h,\varsigma_h,x_h),(v_h,\tau_h,y_h)]& :=(\gamma \nabla w_h - \varsigma_h,\nabla v_h -\gamma^{-1} \tau_h)_\Omega\\& + \alpha (w_h,v_h)_\omega 
+ \beta (\varsigma_h \cdot n, \tau_h \cdot n)_{\partial \Omega} \\
& - (\nabla \cdot \tau_h,x_h)_\Omega -(\nabla \cdot \varsigma_h, y_h)_\Omega.
\end{align*}
Observe that the perturbed boundary condition leads to the appearance of an additional penalty term on the boundary flux (last term of the second line).
The system \eqref{eq:EL1}-\eqref{eq:EL3} may then be written, find $u_h, \sigma_h, z_h \in \mathcal{V}$, where $\mathcal{V}:= V_k\times RT_p \times X_p$, such that
\begin{equation}\label{eq:compact}
A_\gamma[(u_h,\sigma_h,z_h),(v_h,\tau_h,y_h)] = (q,v_h)_\omega + (f,y_h)_\Omega \quad \forall (v_h, \tau_h, y_h) \in \mathcal{V}.
\end{equation}
We can prove the following inf-sup stability estimate
\begin{proposition}\label{prop:infsup}
There exists $c_0>0$, $\beta_0>0$ such that, if $\beta \leq \beta_0 h$ then for all $(w_h, \varsigma_h , x_h) \in \mathcal{V}$ there holds
\[
c_0 |||w_h,\varsigma_h,x_h|||_\beta \leq \sup_{(v_h,\tau_h ,y_h) \in \mathcal{V}\setminus{0}}\frac{A_\gamma[(w_h,\varsigma_h,x_h),(v_h,\tau_h,y_h)]}{|||v_h,\tau_h,y_h|||_\beta}.
\]
\end{proposition}
\begin{proof}
Following the steps of the proof of \cite[Proposition 2.2]{BLO18} it is straightforward to show that for a certain $\varsigma^* \in RT_p$ and $x^* \in X_p$ the following inequalities holds for $\beta=0$, 
\[
 |||w_h,\varsigma_h,x_h|||_0^2 \lesssim A_\gamma[(w_h,\varsigma_h,x_h),(w_h,\varsigma^*,x^*)]
\]
and
\[
|||w_h,\varsigma^*,x^*|||_0 \lesssim |||w_h,\varsigma_h,x_h|||_0.
\]
Now take $v_h =\xi w_h$, $\tau_h = 0$ and $y_h = \xi w_h$, for some $\xi>0$ to be fixed. For this test function we see that
\[
\xi \|\gamma^{\frac12} \nabla w_h\|_\Omega^2 + \xi \alpha \|w_h\|_\omega^2 -\xi (\varsigma_h \cdot n, w_h)_{\partial \Omega} = A_\gamma[(w_h,\varsigma_h,x_h),(\xi w_h,0,\xi w_h)].
\]
Introducing the Poincar\'e inequality, with $\gamma_{min} := \min_{x \in \Omega} \gamma(x)$,
\[
C_P^2 \gamma_{min}\|w_h\|_{H^1(\Omega)}^2 \leq  \|\gamma^{\frac12} \nabla w_h\|_\Omega^2 +  \alpha \|w_h\|_\omega^2,
\]
we see that
\[
\xi C_P^2 \gamma_{min}\| w_h\|_{H^1(\Omega)}^2  -\xi (\varsigma_h \cdot n, w_h)_{\partial \Omega} \leq A_\gamma[(w_h,\varsigma_h,x_h),(\xi w_h,0,\xi w_h)].
\]
By the Cauchy-Schwarz inequality followed by a trace inequality and an arithmetic-geometric inequality we see that
\[
\xi (\varsigma_h \cdot n, w_h)_{\partial \Omega} \leq \xi C_T^2 C_P^{-2} \gamma_{min}^{-1} \|\varsigma_h \cdot n\|_{\partial \Omega}^2 + \xi \frac12 C_P^2 \gamma_{min} \| w_h\|_{H^1(\Omega)}^2.
\]
Hence
\begin{equation}\label{eq:inter_step}
\frac12 \xi C_P \gamma_{min}\| w_h\|_{H^1(\Omega)}^2  -\xi C_T^2 C_P^{-2} \gamma_{min}^{-1}  \|\varsigma_h \cdot n\|_{\partial \Omega}^2 \leq A_\gamma[(w_h,\varsigma_h,x_h),(\xi w_h,0,\xi w_h)].
\end{equation}
It follows that for sufficiently small $\xi$ we have when $\beta>0$,
\[
|||w_h,\varsigma_h,x_h|||_\beta^2 \lesssim A_\gamma[(w_h,\varsigma_h,x_h),((1+\xi) w_h,\varsigma^*,x^*+\xi w_h)].
\]
Note that since $\beta>0$ the negative term in the left hand side of \eqref{eq:inter_step} can be controlled by the term $\beta \|\varsigma_h \cdot n\|_{\partial \Omega}^2$, if $\xi$ is chosen so that $\xi C_T^2 C_P^{-2} \gamma_{min}^{-1} < \beta$.
To conclude we must show that
\begin{align*}
|||(1+\xi) w_h,\varsigma^*,x^* + \xi w_h|||_\beta  & \lesssim |||w_h,\varsigma_h,x_h|||_0 +  \beta^{\frac12} \| w_h\|_{H^1(\Omega)}+ \beta^{\frac12} \|\varsigma^*\cdot n\|_{\partial \Omega}\\
& \lesssim |||(w_h,\varsigma_h,x_h|||_\beta.
\end{align*}
The last inequality follows by the assumption on $\beta$ by recalling from \cite{BLO18} that
\[
\|\varsigma^*\cdot n\|_{\partial \Omega} \lesssim \|\varsigma_h\cdot n\|_{\partial \Omega} + \|h^{-1} x_h\|_{\partial \Omega}.
\]
Therefore
\[
\beta^{\frac12} \|\varsigma^*\cdot n\|_{\partial \Omega} \lesssim \beta^{\frac12}\|\varsigma_h\cdot n\|_{\partial \Omega} + \beta_0 \|h^{-\frac12} x_h\|_{\partial \Omega}\lesssim |||(w_h,\varsigma_h,x_h|||_\beta.
\]
\squareproof
\end{proof}
\begin{remark}
Note that the constant $c_0$ of the bound in Proposition \ref{prop:infsup} is not independent of $h$, since $\beta = O(h)$. Even if the coefficient could be made independent on $h$, by changing the scaling in $h$ on the boundary of the dual variable, this would not imply stability of the original problem, since $\beta = O(1)$ implies the  nonconsistent perturbation of a boundary condition.
\end{remark}
\section{Conditional stability estimates}
Before proving error estimates we need to introduce the stability estimates for the unique continuation problem that the estimates will be based on. The following theorem holds
\begin{theorem}\label{thm:cond_stab}
Let $B \subset \subset \Omega$ and $\omega \subset \Omega$ be two connected sets with non-zero $d$-measure, then there exists $C>0$ and $\tau_0 \in (0,1)$ such that for all $v \in H^1(\Omega)$
\[
\|v\|_B \leq C (\|v\|_\omega + \|\nabla \cdot (\gamma \nabla v)\|_{H^{-1}(\Omega)})^{\tau_0} (\|v\|_\Omega + \|\nabla \cdot (\gamma \nabla v)\|_{H^{-1}(\Omega)})^{(1-\tau_0)}, 
\]
If in addition $|\nabla \ln (\gamma)|$ is sufficiently small then there exists $\tau_1 \in (0,1)$ such that for all $v \in H^1(\Omega)$
\[
\|v\|_{H^1(B)} \leq C (\|v\|_\omega + \|\gamma^{-1}\nabla \cdot (\gamma \nabla v)\|_{H^{-1}(\Omega)})^{\tau_1} (\|v\|_\Omega + \|\gamma^{-1} \nabla \cdot (\gamma \nabla v)\|_{H^{-1}(\Omega)})^{(1-\tau_1)}.
\]
\end{theorem}
\begin{proof}
The first claim is proven in \cite[Theorem 5.1]{ARRV09}.

For the second claim we will use \cite[Corollary 2]{BNO20} where the inequality was proven in the case of an advection--diffusion equation. Set $b = \nabla \ln(\gamma)$. Then observe that for any ball $B_0$ in $\Omega$, the problem 
\[
-\Delta w -  b \cdot \nabla w = \gamma^{-1} f \mbox{ in } B_0
\]
with $w = 0$ on $\partial B_0$ is well posed under the assumptions (where the maximum allowed size of $b$ will depend on the Poincar\'e inequality on the ball $B_0$). Then the result of \cite[Corollary 2]{BNO20} can be shown to hold and we conclude using a classical propagation of smallness argument \cite{Rob98} (see also \cite[Theorem 5.1]{ARRV09}).
\squareproof
\end{proof}
\subsection{Error estimates}
To prove error estimates we proceed in two steps. First we show that for a suitably chosen $\beta$, the error in the norm $|||\cdot,\cdot,\cdot|||_\beta$ converges with a similar rate as one would expect from a well-posed problem. This however does not imply the convergence in any norm uniform in the mesh-size $h$. Indeed without any further stability information of the problem this convergence does not yield any useful error estimate. Instead we must apply the result of Theorem \ref{thm:cond_stab}, which will allow us to deduce error estimates with a rate in the interior of the domain. We will assume that we only have access to $\tilde \gamma = \gamma + \delta \gamma$, $\tilde q = q + \delta q$ and $\tilde f = f + \delta f$. Here the perturbations are assumed to satisfy $\delta q \in L^2(\omega)$, $\delta f \in L^2(\Omega)$ and $\delta \gamma \in L^{\infty}(\Omega)$. To simplify the notation we let 
\[
\delta = \alpha^{\frac12} \|\delta q\|_\omega + \|\delta f\|_\Omega + \frac{\delta \gamma}{\tilde \gamma} \|\nabla u\|_{\Omega}. 
\]

Our approximation may then be written find $\tilde u_h , \tilde \sigma_h, \tilde z_h \in \mathcal{V}$ such that
\begin{equation}\label{eq:pert_compact}
A_{\tilde \gamma}[(\tilde u_h,\tilde \sigma_h,\tilde z_h),(v_h,\tau_h,y_h)] = (\tilde q,v_h)_\omega + (\tilde f,y_h)_\Omega \quad \forall (v_h , \tau_h, y_h) \in \mathcal{V}.
\end{equation}

First note that the formulation \eqref{eq:compact} is consistent in the sense that for the exact solution $u$, $\sigma = \gamma \nabla u$ there holds
\begin{align}\nonumber
A_{\tilde \gamma}[(u,\sigma,0),(v_h,\tau_h,y_h)] & =  (f,y_h)_\Omega+ \beta (\sigma \cdot n, \tau_h)_{\partial \Omega} + (q,v_h)_\omega \\
&+ (\delta \gamma \nabla u,\nabla v_h -{\tilde \gamma}^{-1} \tau_h)_\Omega \quad \forall (v_h , \tau_h, y_h) \in \mathcal{V}.\label{eq:consist}
\end{align}
Combining \eqref{eq:compact} and \eqref{eq:pert_compact} we get the perturbed Galerkin orthogonality
\begin{align}\nonumber
A_{\tilde \gamma}[(\tilde u_h - u,\tilde \sigma_h - \sigma,\tilde z_h),(v_h,\tau_h,y_h)] &= \alpha (\delta q,v_h)_\omega + (\delta f,y_h)_\Omega - \beta (\sigma \cdot n, \tau_h \cdot n)_{\partial \Omega}\\
&- (\delta \gamma \nabla u,\nabla v_h -{\tilde \gamma}^{-1} \tau_h)_\Omega \quad \forall (v_h , \tau_h, y_h) \in \mathcal{V}.\label{eq:galortho}
\end{align}
Let $\pi_h: L^2(\Omega) \mapsto X_p$ denote the standard $L^2$-projection on $X_p$, and $i_h: H^1(\Omega) \mapsto V_k$ some $H^1$-stable interpolant with optimal approximation properties. Also let $r_h: [H^1(\Omega)]^d \mapsto RT_p$ denote the standard interpolant associated with the Raviart-Thomas space \cite{RT77}. Note in particular that for any face $F$ in the mesh $r_h v\vert_F = \pi_F v$ where $\pi_F v$ denotes the $L^2$-projection of $v\vert_F$ onto the space $\mathbb{P}_p(F)$.
It is straightforward to prove the following interpolation error estimate for these spaces
\begin{equation}\label{eq:interpol}
|||(u - i_h u, \sigma - r_h \sigma, 0)|||_0 \leq C(h^k|u|_{H^{k+1}(\Omega)} + h^{p+1} \|\sigma\|_{[H^{p+1}(\Omega)]^d}).
\end{equation}
Noting that by the definition of the Raviart-Thomas interpolant there holds
\[
(\nabla \cdot (\sigma - r_h \sigma), y_h)_\Omega = 0 \mbox{ for all } y_h \in X_p
\]
we may apply the Cauchy-Schwarz inequality to show the continuity
\begin{equation}\label{eq:continuity}
A_{\tilde \gamma}[(i_h u - u,r_h \sigma - \sigma,0),(v_h,\tau_h,y_h)] \leq |||(u - i_h u, \sigma - r_h \sigma, 0)|||_0 |||(v_h, \tau_h, y_h)|||_0.
\end{equation}
Notice that $\beta=0$ can be taken in the norms of the right hand side since $((r_h \sigma - \sigma) \cdot n, \tau_h \cdot n)_F=0$ for all faces $F \subset \partial \Omega$.

Using the stability of Proposition \ref{prop:infsup} together with the bounds \eqref{eq:galortho}, \eqref{eq:continuity} and \eqref{eq:interpol} it is straighforward to prove the following error estimate in the $|||\cdot,\cdot,\cdot|||_\beta$ norm.
\begin{proposition}\label{prop:error1}
Let $u$ be the solution of the unique continuation problem \eqref{eq:UCf} and let $\tilde u_h \in V_k$, $\tilde \sigma_h \in RT_p$, $z_h \in X_p$ be the solution of \eqref{eq:pert_compact} then there holds
\[
|||(u - \tilde u_h, \sigma - \tilde \sigma_h, \tilde z_h)|||_\beta \leq C(h^k|u|_{H^{k+1}(\Omega)} + h^{p+1} \|\sigma\|_{[H^{p+1}(\Omega)]^d} + \beta^{\frac12} \|\sigma \cdot n\|_{\partial \Omega} + \delta).
\]
\end{proposition}
\begin{proof}
The error can be decomposed into the interpolation error and a discrete error, 
\[
u - \tilde u_h = u - i_h u + \underbrace{i_h u - \tilde u_h}_{e_h},
\]
\[
\sigma - \tilde \sigma_h = \sigma - r_h \sigma + \underbrace{r_h \sigma - \tilde \sigma_h}_{\varsigma_h}.
\]
Recalling \eqref{eq:interpol} it is enough to prove the bound for $e_h$ and $\varsigma_h$. Applying the stability of Proposition \ref{prop:infsup} we have
\[
c_0 |||e_h,\varsigma_h,\tilde z_h|||_\beta \leq \sup_{(v_h,\tau_h ,y_h) \in \mathcal{V}\setminus \{0\}}\frac{A_{\tilde \gamma}[(e_h,\varsigma_h,\tilde z_h),(v_h,\tau_h,y_h)]}{|||v_h,\tau_h,y_h|||_\beta}.
\]
Using the perturbed Galerkin orthogonality we see that
\begin{align}\label{eq:pert_eqA}
A_{\tilde \gamma}[(e_h,\varsigma_h,\tilde z_h),(v_h,\tau_h,y_h)] &= A_{\tilde \gamma}[(i_h u - u,r_h \sigma - \sigma,0),(v_h,\tau_h,y_h)] \\
& -\alpha (\delta q,v_h)_\omega - (\delta f,y_h)_\Omega + \beta (\sigma \cdot n, \tau_h \cdot n)_{\partial \Omega}\\
& + (\delta \gamma \nabla u,\nabla v_h -{\tilde \gamma}^{-1} \tau_h)_\Omega.
\end{align}
Applying the Cauchy-Schwarz inequality to the last four terms of the right hand side we get
\begin{align*}
-\alpha (\delta q,v_h)_\omega & - (\delta f,y_h)_\Omega  + \beta (\sigma \cdot n, \tau_h \cdot n)_{\partial \Omega}  + (\delta \gamma \nabla u,\nabla v_h -{\tilde \gamma}^{-1} \tau_h)_\Omega\\  & \leq (\alpha^{\frac12} \|\delta q\|_\omega + \|\delta f\|_\Omega + \beta^{\frac12} \|\sigma \cdot n\|_{\partial \Omega} + \delta \gamma \tilde \gamma^{-\frac12} \|\nabla u\|_\Omega) \\ & \times (\alpha^{\frac12} \|v_h\|_\omega + \|y_h\|_\Omega+ \beta^{\frac12} \|\tau_h \cdot n\|_{\partial \Omega} + \|\tilde \gamma^{\frac12} \nabla v_h -{\tilde \gamma}^{-\frac12} \tau_h\|_\Omega).
\end{align*}
Using the Poincar\'e inequality $\|y_h\|_\Omega \leq C \|y_h\|_{1,h}$ we then obtain
\begin{align*}
-\alpha (\delta q,v_h)_\omega & - (\delta f,y_h)_\Omega  + \beta (\sigma \cdot n, \tau_h \cdot n)_{\partial \Omega} + (\delta \gamma \nabla u,\nabla v_h -{\tilde \gamma}^{-1} \tau_h)_\Omega\\ 
& \leq C (\alpha^{\frac12} \|\delta q\|_\omega + \|\delta f\|_\Omega + \beta^{\frac12} \|\sigma \cdot n\|_{\partial \Omega}+ \delta \gamma \tilde \gamma^{-\frac12} \|\nabla u\|_\Omega) |||v_h,\tau_h,y_h|||_\beta\\
& \leq C (\delta + \beta^{\frac12} \|\sigma \cdot n\|_{\partial \Omega}) |||v_h,\tau_h,y_h|||_\beta.
\end{align*}
Applying the continuity \eqref{eq:continuity} to the first term of the right hand side of \eqref{eq:pert_eqA} we may deduce the bound
\[
c_0 |||e_h,\varsigma_h,\tilde z_h|||_\beta \leq C (|||u - i_h u,\sigma - r_h \sigma ,0|||_0 + \delta + \beta^{\frac12} \|\sigma \cdot n\|_{\partial \Omega}).
\]
The claim follows by applying the approximation bound \eqref{eq:interpol}.
\squareproof
\end{proof}

By inspecting the above error bound we see that the choice leading to the best convergence is $p=k-1$, $\beta = 2 k$ this leads to the following bound.
\begin{corollary}\label{cor:error2}
Let $u \in H^{k+1}(\Omega)$ be the solution of the unique continuation problem \eqref{eq:UCf}. Assume that $\sigma \in [H^{k}(\Omega)]^d$. Let $\tilde u_h \in V_k$, $\tilde \sigma_h \in RT_{k-1}$, $z_h \in X_{k-1}$ be the solution of \eqref{eq:pert_compact}, with $\beta = h^{2 k}$ then there holds
\[
|||(u - \tilde u_h, \sigma - \tilde \sigma_h, \tilde z_h)|||_\beta \leq C(h^k \Phi(u,\sigma) + \delta),
\]
where
\[
\Phi(u,\sigma) :=|u|_{H^{k+1}(\Omega)} +\|\sigma\|_{[H^{k}(\Omega)]^d}.
\]
\end{corollary}
\begin{proof}
Immediate by the bound of Proposition \ref{prop:error1}, the assumptions on $k$ and $p$ and $\beta$ and the trace inequality $\|\sigma \cdot n\|_{\partial \Omega} \leq C \|\sigma\|_{[H^1(\Omega)]^d}$.
\squareproof
\end{proof}

This estimate implies that for smooth solutions without perturbations in data the blowup rate of $\|u - u_h\|_{H^1(\Omega)}$ is contained. Indeed we have the following corollary,
\begin{corollary}\label{cor:H1}
Under the assumptions of Corollary \ref{cor:error2} there holds
\[
\|u - \tilde u_h\|_{H^1(\Omega)} \leq C (\Phi(u,\sigma) + h^{-k} \delta ).
\]
\end{corollary}
\begin{proof}
We observe that by definition
\[
\|u - \tilde u_h\|_{H^1(\Omega)} \leq \beta^{-\frac12} |||(u - \tilde u_h,0, 0)|||_\beta.
\]
The claim follows by applying the error bound of Corollary \ref{cor:error2} and recalling that $\beta = O(h^{2k})$.
\squareproof
\end{proof}

After these preliminary considerations we are now ready to prove an interior error estimate.
\begin{theorem}\label{thm:error_cond}
Assume that the hypothesis of Corollary \ref{cor:error2} are satisfied. Also assume that $h<1$. Let $B$ and $\omega$ be the subsets of $\Omega$ defined in Theorem \ref{thm:cond_stab} then there holds
\[
\|u - \tilde u_h\|_{B} \leq C h^{k \tau_0} (\Phi(u,\sigma) + h^{-k} \delta)  
\]
and if $|\nabla \ln(\gamma)|$ is small enough,
\[
\|u - \tilde u_h\|_{H^1(B)} \leq C h^{k \tau_1} (\Phi(u,\sigma) + h^{-k} \delta).
\]
\end{theorem}
\begin{proof}
We only detail the first inequality. The proof of the second is similar using the second inequality of Theorem \ref{thm:cond_stab}.
Let $e = u-\tilde u_h \in H^1(\Omega)$. Then by Theorem \ref{thm:cond_stab} there holds
\[
\|e\|_B \leq C (\|e\|_\omega + \|\nabla \cdot (\gamma \nabla e)\|_{H^{-1}(\Omega)})^{\tau_0} (\|e\|_\Omega + \|\nabla \cdot (\gamma \nabla e)\|_{H^{-1}(\Omega)})^{(1-\tau_0)}.
\]
Recalling the bounds of Corollaries \ref{cor:error2} and \ref{cor:H1} we see that
\[
\|e\|_\omega \leq C  (h^k \Phi(u,\sigma) + \delta)
\]
and 
\[
\|e\|_{\Omega} \leq C (\Phi(u,\sigma) + h^{-k} \delta ).
\]
It remains to bound $\|\nabla \cdot (\gamma \nabla e)\|_{H^{-1}(\Omega)}$.  By definition
\[
\|\nabla \cdot (\gamma \nabla e)\|_{H^{-1}(\Omega)} = \sup_{\substack{v \in H^1_0(\Omega) \\ \|v\|_{H^1(\Omega)} = 1}} (\gamma \nabla e, \nabla v)_\Omega.
\]
Considering the term in the right hand side we see that, with $\varsigma = \sigma - \tilde \sigma_h$,
\[
(\gamma \nabla e, \nabla v)_\Omega = \underbrace{(\gamma \nabla e - \varsigma, \nabla v)_\Omega}_{I_1} + \underbrace{(\varsigma, \nabla v)_\Omega}_{I_2} = I_1+ I_2.
\]
For the first term we see that by the Cauchy-Schwarz inequality and by recalling the definition of $|||(\cdot,\cdot,\cdot)|||_\beta$ and the bound of Corollary \ref{cor:error2} there holds
\[
I_1 \leq \|\gamma\|_{L^\infty(\Omega)}^{\frac12}  \|\gamma^{\frac12} \nabla e - \gamma^{-\frac12}\varsigma\|_\Omega \leq C \|\gamma\|_{L^\infty(\Omega)}^{\frac12} (h^k \Phi(u,\sigma) + \delta).
\]
To bound the term $I_2$ we integrate by parts and use the equation to obtain
\[
I_2 = (\nabla \cdot \varsigma, v)_\Omega = -(f,v)_\Omega - (\nabla \cdot \tilde \sigma_h, v)_\Omega.
\]
Recalling that 
\[
(\tilde f + \nabla \cdot \tilde \sigma_h,x_h)_\Omega = 0 \mbox{ for all } x_h \in X_0
\]
we have, with $\pi_0$ denoting the $L^2$-projection onto $X_0$ and using the elementwise Poincar\'e inequality to obtain $\|v - \pi_0 v\|_T \leq C h_T \|\nabla v\|_T$, 
\[
I_2 = (\nabla \cdot \varsigma, v - \pi_0 v)_\Omega + (\delta f, v)_\Omega \leq C (\|h \nabla \cdot \varsigma\|_\Omega + \delta).
\]
Noting that $\|h \nabla \cdot \varsigma\|_\Omega \leq |||(0,\varsigma,0)||||_\beta$ we conclude that
\[
I_2 \leq C(|||(0,\varsigma,0)||||_\beta + \delta) \leq C (h^k \Phi(u,\sigma) + \delta).
\]
Collecting terms we see that
\[
\|e\|_B \leq C (h^k \Phi(u,\sigma) + \delta)^{\tau_0} (\Phi(u,\sigma) + h^{-k} \delta + h^k \Phi(u,\sigma) + \delta)^{(1-\tau_0)}
\]
Using that $(h^k \Phi(u,\sigma)+ \delta)^{\tau_0} = h^{k \tau_0}(\Phi(u,\sigma) + h^{-k} \delta)^{\tau_0}$ and since $h<1$,
\[
(\Phi(u,\sigma) + h^{-k} \delta)^{\tau_0} (\Phi(u,\sigma) + h^{-k} \delta + h^k \Phi(u,\sigma) + \delta)^{(1-\tau_0)} \leq 2 (\Phi(u,\sigma) + h^{-k} \delta)
\]
from which the claim follows.

For the second estimate the only difference is that we need to apply the second bound of Theorem \eqref{thm:cond_stab}, which in particular implies that we must bound
\[
\|\gamma^{-1}\nabla \cdot (\gamma \nabla e)\|_{H^{-1}(\Omega)} = \sup_{\substack{v \in H^1_0(\Omega) \\ \|v\|_{H^1(\Omega)} = 1}} (\gamma \nabla e, \nabla (\gamma^{-1} v))_\Omega.
\]
However since $\gamma \in W^{1,\infty}(\Omega)$ and strictly positive we have $$\|\nabla (\gamma^{-1} v)\|_{L^2(\Omega)} \leq C_{\gamma_{min}} \|\gamma\|_{W^{1,\infty}(\Omega)} \|v\|_{H^1(\Omega)}$$ and we can therefore proceed as before, but with $\gamma^{-1} v$ in the place of $v$.
\end{proof}
\begin{remark}
It follows from Theorem \ref{thm:error_cond} that refinement should be stopped when $h^k \approx \delta/\Phi(u,\sigma)$. This can be built in to the method by taking 
$\beta = \max(h, h_0)^{2k}$, where $h_0 = (\delta/\Phi(u,\sigma))^{\frac{1}{k}}$. Then the following bound holds for all $h>0$. 
\[
\|u - \tilde u_h\|_{B} \leq C \max(h,h_0)^{k \tau_0} \Phi(u,\sigma). 
\]
\end{remark}
\section{Applying the method for the approximation of $\gamma$}\label{sec:steep}
To use the above method for the reconstruction of $\gamma$ we simply minimize the functional $\mathcal{L}_\gamma$, \eqref{eq:lagrange} with respect to $\gamma$. Typically to enhance stability some additional stabilization of the $\gamma$ variable must be added. Below we will use Tikhonov regularizaton on the $H^1$-seminorm. The regularized Lagrangian takes the form
 \begin{equation}\begin{aligned}
  \mathcal{L}_\gamma(\sigma_h,u_h,z_h) &:= \frac12 \|\gamma^{-\frac12} (\gamma \nabla u_h - \sigma_h)\|_\Omega^2 + \frac12 \alpha \|u_h - q\|_\omega^2 - (\nabla \cdot \sigma,z_h)_\Omega - (f,z_h)_\Omega\\
   &+ \frac12 \|h^{\frac12} \nabla \gamma\|_\Omega^2 + \frac12 \beta \|\sigma_h \cdot n\|_{\partial \Omega}
 \label{eq:Lagreg}
 \end{aligned}
 \end{equation}
 and the optimization then reads
 \[
 \gamma = \mbox{argmin}_{\tilde \gamma \in \Gamma_l}  \mathcal{L}_{\tilde \gamma}(\sigma_h,u_h,z_h).
 \]
Here we introduced the space $\Gamma_l$ for the approximation of the coefficient $\gamma$ and below we will always assume that $\Gamma_l = V_l$ or $\Gamma_l = X_l$ with $l\leq k$.

The formulation \eqref{eq:compact} can then be used for the construction of the gradient for instance when using the well-known steepest descent algorithm. This leads to a robust, but slow reconstruction method.  We give the algorithm in Algorithm \ref{alg:steep} below, for one data point $q$. However in practice this appears not to be enough. Several sampled solutions must be available and preferably both data on $u$ in the bulk and on $\sigma \cdot n$ on the boundary to obtain a stable reconstruction. 
\begin{algorithm}[h!]
\begin{flushleft}
Assume   $\gamma^0 \in \Gamma_l$ , $l \leq k$, and associated $q$, to be given, fix $TOL>0$ and, for $n\geq 1$, perform:
\end{flushleft}
\begin{itemize}
\item Compute $u^{n}$, $\sigma^{n}$ solution to 
\begin{equation}\label{eq:compact_iter}
A_{\gamma^{n-1}}[(u^{n},\sigma^{n},z_h),(v_h,\tau_h,y_h)] = (q,v_h)_\omega + (f,y_h)_\Omega \quad \forall v_h,\, \tau_h ,\, y_h \in \mathcal{V}.
\end{equation}
\item Find $g^n \in \Gamma_l$ such that
\[
(g^n,v)_\Omega + (h \nabla g^n,\nabla v)_\Omega = (|\nabla u_h^n|^2 - \gamma^{-1} |\sigma_h^n|^2, v)_\Omega \mbox{ for all } v
\in \Gamma_l.
\]
\item Let $s^n>0$ be a steplength that can be fixed or determined using a line search.
\item Update the coefficient.
\[
(\gamma^{n},v)_\Omega = (\gamma^{n-1},v)_\Omega + (s^n g^n,v)_\Omega. 
\]
\item Repeat until $\|g^n\|_\Omega \leq TOL$.
\end{itemize}
\caption{Steepest descent algorithm for coefficient recovery}\label{alg:steep}
\end{algorithm}
We recall that it is thanks to the robustness of the unique continuation method discussed above that the problem \eqref{eq:compact_iter} is solvable for all $\gamma^n$.
\section{Numerical examples}
To illustrate the theory of the previous section we present two numerical experiments. In both cases we consider the unit disc, $\Omega = \{(x,y) \in \mathbb{R}^2: x^2+y^2<1\}$. The interior data is given in the zone $\omega = \{(x,y) \in \mathbb{R}^2: 0.75 < x^2+y^2<1; x < 0.5\}$. Only unperturbed data were considered. We chose $\beta = 10^{-3} h^{2k}$ in the computations below, but similar results were obtained for $\beta = 0$ so this regularization was not essential for the present computations.
\subsection{Unique continuation}
We consider the formulation \eqref{eq:compact} with $\alpha=1000.$ and the exact solution given by the solution to 
\[
\nabla \cdot \gamma \nabla u=0 \mbox{ in } \Omega
\]
with $u=\sin(3 \theta)$ on $\partial \Omega$. The reference solution was computed using conforming finite elements of order $4$ on a mesh with three times more elements on the boundary of the disc than the mesh used for the continuation. We considered three combinations of finite elements, $V_k \times RT_{k-1} \times X_{k-1}$, with $k \in \{1,2,3\}$. In Fig. \ref{fig:UC_conv} we report the convergence of the local $L^2$-errors, in $\Omega_-:= \{(x,y) \in \Omega:\, x \leq 0 \}$,
$\|u - u_h\|_{\Omega_-}$. The graph for $k=1$ is indicated with circle markers, $k=2$ with triangle markers and $k=3$ with square markers.
Three reference lines are given. The dotted line is $y = 0.01 x^{0.45}$, the dashed line is $y = 0.001 x^{0.9}$ and the filled line without markers is $y = 0.0001 x^{1.35}$. Observe that the lowest order approximation, $k=1$ converges slightly faster than $0.45$, the case $k=2$ has much smaller error than for $k=1$, but the observed rate is similar. Finally in the case $k=3$ the convergence order is approximately $1.35$. We conclude that the power in the conditional stability estimate appears to be approximately $\tau_0 = 0.45$ for this case. We note that it is possible that the curved boundary affects the convergence in the higher order cases.

\begin{figure}[!http]
\begin{center}
\includegraphics[scale=0.4]{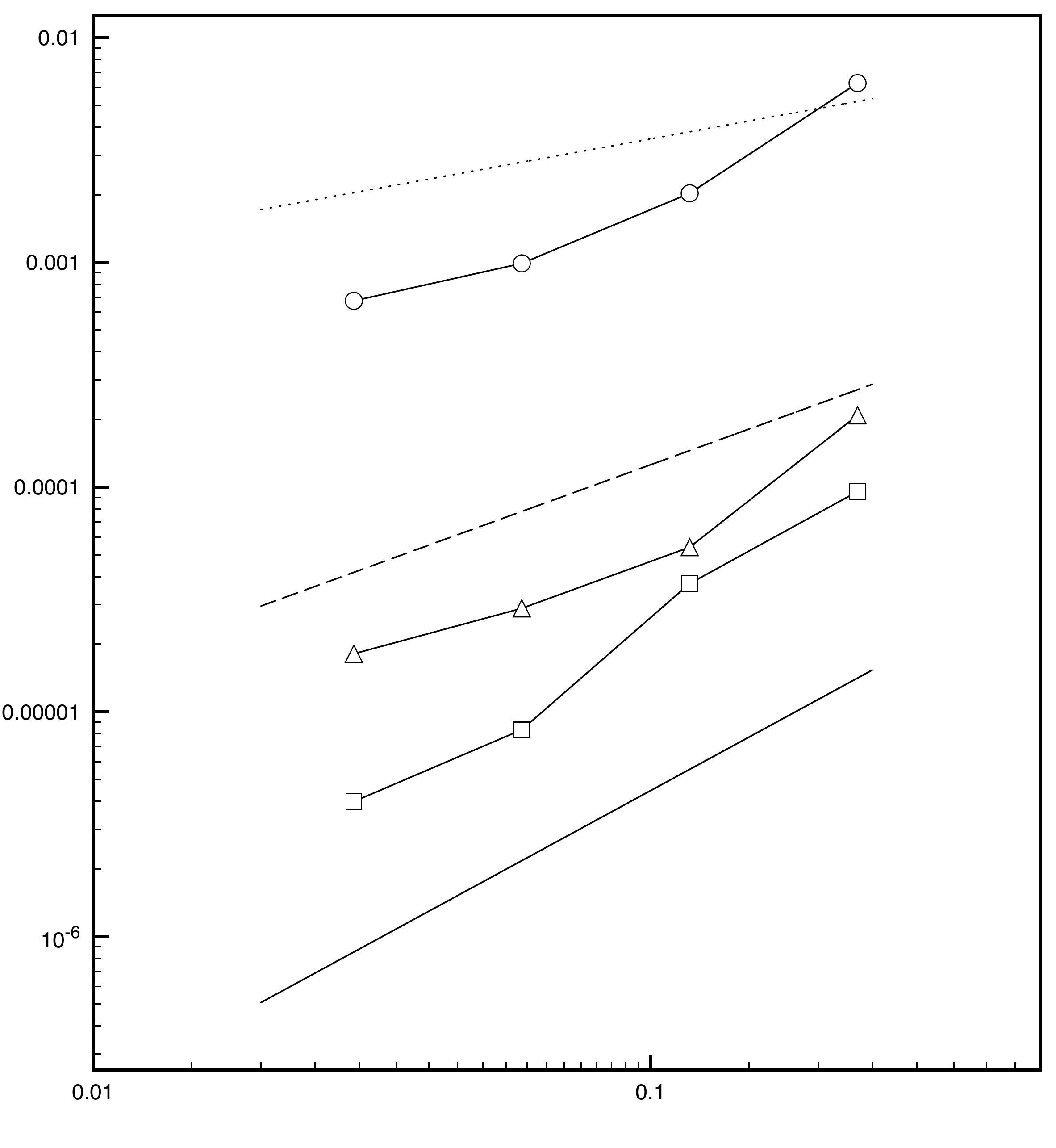}
\caption{Local $L^2$-norm error $\|u - u_h\|_{\Omega_-}$, plotted againsy $h$. The dotted line is  $y = 0.01 x^{0.45}$ the dashed line is $y = 0.001 x^{0.9}$ and the filled line without markers are is $y = 0.0001 x^{1.35}$.}\label{fig:UC_conv}
\end{center}
\end{figure}
\subsection{Reconstruction of the diffusivity coefficient}
In this example we will use partial data and reconstruct the diffusivity coefficient $\gamma = 1 + \exp(-(x+0.3)^{2}-(y-0.3)^{2})$, starting from the initial guess $\gamma^0=1$. In addition to data in $\omega$ we add the Neumann data on the part of the boundary with the polar coordinates $r=1$ and $0\leq \theta \leq 3 \pi/2$ by imposing it strongly on the $\sigma_h$ variable. Observe that the sets of boundary and interior data do not match on the boundary. We consider two data sets corresponding to the solution given by 
\[
\nabla \cdot \gamma \nabla u=0 \mbox{ in } \Omega
\]
with $u= \cos(k_1 \theta) + \sin(k_2 \theta)$ on $\partial \Omega$ and $(k_1,k_2) \in \{(2,1), \,(3,2) \}$. We compute an approximate solution $\gamma_h \in V_1$ on six consecutive meshes, starting with $h \approx 0.168$ and then dividing the mesh size by two in every refinement. The Tikhonov type regularization $\frac12 \|h^{\frac12} \nabla \gamma_h\|^2_\Omega$ (on exactly this form, i.e. with parameter $1$ in the Euler-Lagrange equations) was added to the functional. Note that in the steepest descent algorithm above the regularizing terms is treated implicitly. On each mesh level the approximation was computed using $300$ steepest descent iterations as defined in the section \ref{sec:steep} and fixed step length $s=0.8$. The spaces used for $u_h,\, \sigma_h,\, z_h$ were chosen as $V_2$, $RT_1$, $X_1$.  The convergence 
behavior is reported in Fig. \ref{fig:cald_conv} where the $L^2$-norm error, $\|\gamma - \gamma_h\|_\Omega$ is plotted against $|\log(h)|^{-1}$. The dotted lines are reference curves on the form $c |\log(h)|^{-0.5}$, with $c=0.18$ (top) and $c=0.15$ (bottom). The experimental order of convergence turns out to be very close to $\|\gamma - \gamma_h\|_\Omega = 0.1625 |\log(h)|^{-0.5}$. Attempts to perform the reconstruction with $k=1$ failed due to convergence to an unphysical minimum on the coarse mesh. The use of $k=3$ performed similarly as $k=2$, without any improved accuracy.
\begin{figure}[!http]
\begin{center}
\includegraphics[scale=0.4]{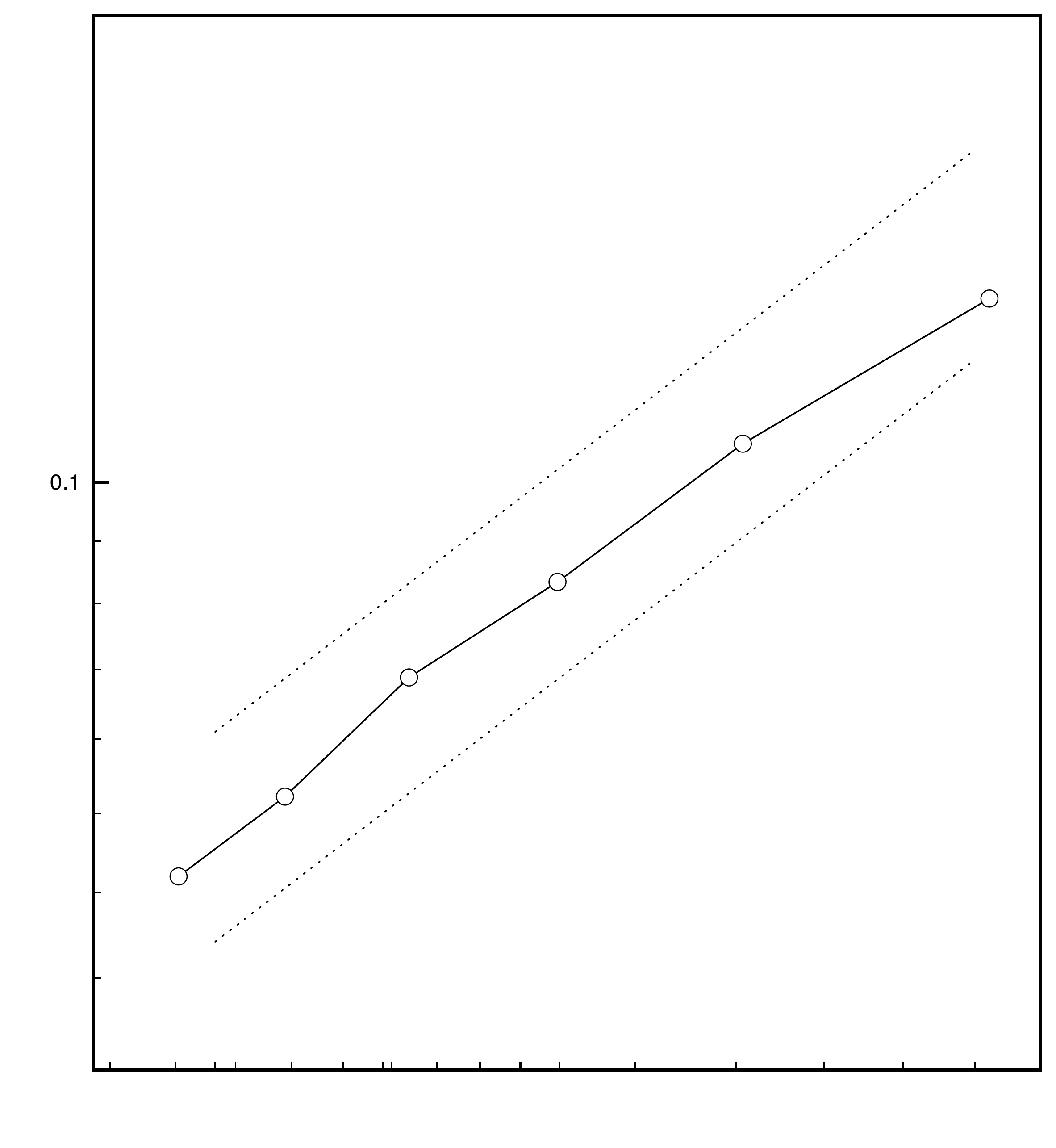}
\caption{$L^2$-norm error, $\|\gamma - \gamma_h\|_\Omega$ is plotted against $|\log(h)|^{-1}$. The dotted lines are reference curves on the form $c |\log(h)|^{-0.5}$. The coarsest mesh has $h=0.168$ and the finest mesh $h=0.053$.}\label{fig:cald_conv}
\end{center}
\end{figure}
\section{Conclusion}
In this paper we extended the primal dual method introduced in \cite{BLO18}, for the approximation of the elliptic Cauchy problem, to the case of unique continuation from interior data. We showed that the method is a discrete realization of the Kohn-Vogelius method and that error estimates supported by conditional stability estimates can be achieved if the boundary condition of one of the sub problems is perturbed. The results were illustrated numerically for unique continuation and for the reconstruction of the diffusion coefficient using partial data. In both cases the convergence rates appear to match those predicted by theoretical stability estimates. The local error for the unique continuation problem had H\"older type convergence, $O(h^\alpha)$,$ \alpha \in (0,1)$. The convergence of the error in the diffusion coefficient when reconstructed with partial data had  logarithmic convergence, $O(|\log(h)|^\alpha)$, $\alpha \in (-1,0)$. 

\section*{Acknowledment} The author was funded by the EPSRC grants EP/T033126/1 and EP/V050400/1. The author also thanks the $PeC^3$ networks and the organising committee of the {\emph{Peruvian conference on Scientific Computing}} for the invitation to attend the event and their hospitality. 
\bibliographystyle{abbrv}
\bibliography{ref}

\end{document}